\documentclass[english]{article}
\usepackage[T1]{fontenc}
\usepackage[latin9]{inputenc}
\usepackage{babel}
\usepackage{amsmath}
\usepackage{amsthm}
\usepackage{amssymb}
\usepackage[all]{xy}
\usepackage[unicode=true,pdfusetitle,
 bookmarks=true,bookmarksnumbered=false,bookmarksopen=false,
 breaklinks=false,pdfborder={0 0 1},backref=false,colorlinks=false]
 {hyperref}

\makeatletter
\numberwithin{equation}{section}
\numberwithin{figure}{section}
\theoremstyle{plain}
\newtheorem{thm}{\protect\theoremname}
\theoremstyle{plain}
\newtheorem{lem}[thm]{\protect\lemmaname}
\ifx\proof\undefined
\newenvironment{proof}[1][\protect\proofname]{\par
	\normalfont\topsep6\p@\@plus6\p@\relax
	\trivlist
	\itemindent\parindent
	\item[\hskip\labelsep\scshape #1]\ignorespaces
}{%
	\endtrivlist\@endpefalse
}
\providecommand{\proofname}{Proof}
\fi
\theoremstyle{plain}
\newtheorem{cor}[thm]{\protect\corollaryname}
\theoremstyle{definition}
\newtheorem{defn}[thm]{\protect\definitionname}

\makeatletter
\newcommand{\xyR}[1]{%
\makeatletter
\xydef@\xymatrixrowsep@{#1}
\makeatother
} 

\makeatletter
\newcommand{\xyC}[1]{%
\makeatletter
\xydef@\xymatrixcolsep@{#1}
\makeatother
} 

\makeatother

\providecommand{\corollaryname}{Corollary}
\providecommand{\definitionname}{Definition}
\providecommand{\lemmaname}{Lemma}
\providecommand{\theoremname}{Theorem}

\begin{document}
\title{Free-algebra functors from a coalgebraic perspective}
\author{H. Peter Gumm}
\maketitle
\begin{abstract}
We continue our study of free-algebra functors from a coalgebraic
perspective as begun in \cite{Gumm2019}. Given a set $\Sigma$ of
equations and a set $X$ of variables, let $F_{\Sigma}(X)$ be the
free $\Sigma-$algebra over $X$ and $\mathcal{V}(\Sigma)$ the variety
of all algebras satisfying $\Sigma.$ We consider the question, under
which conditions the $Set$-functor $F_{\Sigma}$ weakly preserves
pullbacks, kernel pairs, or preimages \cite{GS05}.

We first generalize a joint result with our former student Ch. Henkel,
asserting that an arbitrary $Set-$endofunc\-tor $F$ weakly preserves
kernel pairs if and only if it weakly preserves pullbacks of epis.

By slightly extending the notion of derivative $\Sigma'$ of a set
of equations $\Sigma$ as defined by Dent, Kearnes and Szendrei in
\cite{Dent2012}, we show that a functor $F_{\Sigma}$ (weakly) preserves
preimages if and only if $\Sigma$ implies its own derivative, i.e.
$\Sigma\vdash\Sigma'$, which amounts to saying that weak independence
implies independence for each variable occurrence in a term of $\mathcal{V}(\Sigma)$.
As a corollary, we obtain that the free-algebra functor will never
preserve preimages when $\mathcal{V}(\Sigma)$ is congruence modular.

Regarding preservation of kernel pairs, we show that for \emph{n}-permutable
varieties $\mathcal{V}(\Sigma),$ the functor $F_{\Sigma}$ preserves
kernel pairs if and only if $\mathcal{V}(\Sigma)$ is a Mal'cev variety,
i.e. 2-permutable.
\end{abstract}

\section{Introduction}

In his groundbreaking monograph ``Universal Coalgebra \textendash{}
a theory of systems''\cite{Rutten2000} Jan Rutten demonstrated how
all sorts of state based systems could be unified under the roof of
one abstract concept, that of a coalgebra. Concrete system types \textendash{}
automata, transition systems, nondeterministic, weighted, probabilistic
or second order systems \textendash{} can be modeled by choosing an
appropriate functor $F:Set\to Set$ which provides a \emph{type} for
the concrete coalgebras just as, on a less abstract level, signatures
describe types of algebras.

The (co)algebraic properties of the category $Set_{F}$ of all $F-$coalgebras
are very much dependent on certain preservation properties of the
functor $F$. A particular property of $F,$ which has been considered
relevant from the beginning was the preservation of certain weak limits.
In Rutten's original treatise\cite{Rut96}, lemmas and theorems were
marked with an asterisk, if they used the additional assumption that
$F$ \emph{weakly preserves pullbacks}, that is $F$ transforms pullback
diagrams in $Set$ into weak pullback diagrams.

In our lecture notes \cite{Gum99b}, we were able to remove a large
number of asterisks from Rutten's original presentation, and in joint
work with T. Schr\"{o}der \cite{GS05}, we managed to split the mentioned
preservation condition into two separate conditions, \emph{weak preservation
of kernel pairs} and \emph{preservation of preimages}. These properties
were then studied separately with an eye on their structure theoretic
significance.

It is therefore relevant and interesting to classify $Set$ functors
according to the mentioned preservation properties.

We start this paper by showing that for arbitrary $Set$ functors
weak preservation of kernel pairs is equivalent to weak preservation
of pullbacks of epis, a result, which was obtained jointly with our
former master student Ch.~Henkel.

Subsequently, we investigate $Set$ functors $F_{\Sigma}$ which associate
to a set $X$ the free $\Sigma-$algebra over $X.$ It turns out that
(weak) preservation of preimages by $F_{\Sigma}$ can be characterized
utilizing the \emph{derivative} $\Sigma'$ of $\Sigma$, which has
been studied a few years ago by Dent, Kearnes, and Szendrei \cite{Dent2012}.
For arbitrary sets of \emph{idempotent} equations $\Sigma$, they
show that the variety $\mathcal{V}(\Sigma)$ is congruence modular
if and only if $\Sigma\cup\Sigma'$ is inconsistent. Below, we extend
their notion of derivative to arbitrary sets of equations (not necessarily
idempotent) and are able to show that $F_{\Sigma}$ weakly preserves
preimages if and only if $\Sigma\vdash\Sigma'.$

Regarding preservation of kernel pairs, we exhibit an algebraic condition,
which to our knowledge has not been studied before and which appears
to be interesting in its own right. If $F_{\Sigma}$ weakly preserves
kernel pairs, that is if $F_{\Sigma}$ weakly preserves pullbacks
of epis, then for any pair $p,q$ of ternary terms satisfying
\[
p(x,x,y)\approx q(x,y,y)
\]
 there exists a quaternary term $s$ such that

\begin{eqnarray*}
p(x,y,z) & \approx & s(x,y,z,z)\\
q(x,y,z) & \approx & s(x,x,y,z).
\end{eqnarray*}

Applying this to the description of $n-$permutable varieties given
by Hagemann and Mitschke \cite{Hagemann1973}, we find that for an
$n-$permutable variety $\mathcal{V}(\Sigma),$ the functor $F_{\Sigma}$
weakly preserves kernel pairs if and only if $\mathcal{V}(\Sigma)$
is a Mal'cev variety, i.e. there exists a term $m(x,y,z)$ such that
the equations 
\begin{eqnarray*}
m(x,y,y) & \approx & x\\
m(x,x,y) & \approx & y
\end{eqnarray*}
are satisfied.

\section{Preliminaries}

For the remainder of this work we shall denote function application
by juxtaposition, associating to the right, i.e. $fx$ denotes $f(x)$
and $fgx$ denotes $f(g(x)).$ 

If $F$ is a functor, we denote by $F(X)$ the application of $F$
to an object $X$ and by $Ff$ the application of $F$ to a morphism
$f.$

Given a $Set$-functor $F,$ an $F$-coalgebra is simply a map $\alpha:A\to F(A)$,
where $A$ is called the \emph{base set} and $\alpha$ the \emph{structure
map} of the coalgebra $\mathcal{A}=(A,\alpha).$ A \emph{homomorphism}
between coalgebras $\mathcal{A}=(A,\alpha)$ and $\mathcal{B}=(B,\beta)$
is a map $\varphi:A\to B$ satisfying 
\[
\beta\circ\varphi=F\varphi\circ\alpha.
\]

The functor $F$ is called the \emph{type} of the coalgebra. The class
of all coalgebras of type $F$ together with their homomorphisms forms
a category $Set_{F}.$ The structure of this category is known to
depend heavily on several pullback preservation properties of the
functor $F.$ 

A pullback diagram 
\[
\xymatrix{A_{1}\ar[r]^{f_{1}} & C\\
P\ar[u]^{p_{1}}\ar[r]_{p_{2}} & A_{2}\ar[u]_{f_{2}}
}
\]
is called a \emph{kernel pair}, if $f_{1}=f_{2}$ and it is called
a \emph{preimage} if $f_{1}$ or $f_{2}$ is mono.

A functor $F$ is said to \emph{weakly preserve pullbacks} if $F$
transforms each pullback diagram into a weak pullback diagram. $F$
weakly preserving kernel pairs, resp. preimages are defined likewise. 

In order to check whether in the category $Set$ a diagram as above
is a weak pullback, we may argue elementwise: $(P,p_{1},p_{2})$ is
a weak pullback, if for each pair $(a_{1},a_{2})$ with $a_{i}\in A_{i}$
and $f_{1}a_{1}=f_{2}a_{2}$ there exists some $a\in P$ such that
$p_{1}a=a_{1}$ and $p_{2}a=a_{2}.$

Hence, to see that a $Set-$functor $F$ weakly preserves a pullback
$(P,p_{1},p_{2}),$ we must check that for any pair $(u_{1},u_{2})$
with $u_{i}\in F(A_{i})$ and $(Ff_{1})u_{1}=(Ff_{2})u_{2}$ we can
find some $w\in F(P)$ such that$(Fp_{1})w=u_{1}$ and $(Fp_{2})w=u_{2}.$

One easily checks that a \emph{weak preimage} is automatically a \emph{preimage.}
Hence, if $F$ preserves monos, then $F$ \emph{preserves }preimages
if and only if $F$\emph{ weakly preserves }preimages. 

Assuming the axiom of choice, all epis in the category $Set$ are
right invertible, hence $F$ preserves epis. Monos are left-invertible,
except for the empty mappings $\emptyset_{X}:\emptyset\to X$ when
$X\ne\emptyset$. Hence $F$ surely preserves monos with nonempty
domain. 

Most \emph{Set}-functors also preserve monos with empty domain. This
will, in particular, be the case for the free-algebra functor $F_{\Sigma},$
which we shall study in the later parts of this work.

For $Set$-functors $F$ which fail to preserve monos with empty domain,
there is an easy fix, modifying $F$ solely on the empty set $\emptyset$
and on the empty mappings $\emptyset_{X}:\emptyset\to X$, so that
the resulting functor $F^{\star}$ preserves all monos. The details
can be found in \cite{Bar93} or in \cite{Gum05}. This modification
is irrelevant as far as coalgebras are concerned, since it affects
only the empty coalgebra. Yet it allows us to assume from now on,
that $F$ preserves all monos and all epis.

If $f_{1}$ and $f_{2}$ are both injective, then their pullback is
called an intersection. It is well known from \cite{TrnkovaDescr1}
that a set functor automatically preserves nonempty intersections,
and, after possibly modifying it at $\emptyset$ as indicated above,
preserves all finite intersections.

\section{Weak preservation of epi pullbacks}

The following lemma from \cite{GS05} shows that weak pullback preservation
can be split into two separate preservation requirements.
\begin{lem}
For a $Set$-functor $F$ the following are equivalent:
\begin{enumerate}
\item F weakly preserves pullbacks.
\item F weakly preserves kernels and preimages.
\end{enumerate}
\end{lem}
A special case of a preimage is obtained if we consider a subset $U\subseteq A$
as the preimage of $\{1\}$ along its characteristic function $\chi_{U}:A\to\{0,1\}.$
\[
\xymatrix{A\ar[r]^{\chi_{U}} & \{0,1\}\\
U\ar@{^{(}->}[u]\ar[r]^{!_{U}} & \{1\}\ar@{^{(}->}[u]
}
\]
Such preimages are called \emph{classifying, }and we shall later make
use of the following lemma from \cite{GS05}:
\begin{lem}
\label{lem:classifying}A $Set$-functor $F$ preserves preimages
if and only if it preserves classifying preimages.
\end{lem}
In the following section we need to consider the action of a functor
on pullback diagrams where both $f_{1}$ and $f_{2}$ are surjective.
Before stating this, we shall prove a useful lemma, which is true
in every category:
\begin{lem}
\label{lem:sum}Let morphisms $f:A\to C$ and $f_{i}:A_{i}\to C$,
for $i=1,2$ be given, as well as $e_{i}:A_{i}\to A$ with left inverses
$h_{i}:A\to A_{i}$ such that the diagram below commutes. If $(K,\pi_{1},\pi_{2})$
is a weak kernel of $f$ then $(K,h_{1}\circ\pi_{1},h_{2}\circ\pi_{2})$
is a weak pullback of $f_{1}$ and $f_{2}.$
\[
\xymatrix{ &  & A_{i}\ar@<.2pc>[d]^{e_{i}}\ar[drr]^{f_{i}}\ar@(ur,ul)\\
K\ar@<.2pc>[rr]^{\pi_{1}}\ar@<-.2pc>[rr]_{\pi_{2}} & \, & A\ar@<.2pc>[u]^{h_{i}}\ar[rr]_{f} & \, & {\,\,C\,\,}
}
\]
\end{lem}
\begin{proof}
Assuming $(K,\pi_{1},\pi_{2})$ is a weak kernel of $f,$ then setting
$k_{i}:=h_{i}\circ\pi_{i}$ we obtain
\[
f_{1}\circ k_{1}=f_{1}\circ h_{1}\circ\pi_{1}=f\circ\pi_{1}=f\circ\pi_{2}=f_{2}\circ h_{2}\circ\pi_{2}=f_{2}\circ k_{2}.
\]
This shows that $(K,k_{1},k_{2})$ is a candidate for a pullback of
$f_{1}$ with $f_{2}$. Let $(Q,q_{1},q_{2})$ be another candidate,
i.e. 
\[
f_{1}\circ q_{1}=f_{2}\circ q_{2},
\]
 then we obtain 
\[
f\circ e_{1}\circ q_{1}=f_{1}\circ q_{1}=f_{2}\circ q_{2}=f\circ e_{2}\circ q_{2},
\]
which demonstrates that $(Q,e_{1}\circ q_{1},e_{2}\circ q_{2})$ is
a competitor to $(K,\pi_{1},\pi_{2})$ for being a weak kernel of
$f.$ This yields a morphism $q:Q\to K$ with 

\[
\pi_{i}\circ q=e_{i}\circ q_{i}.
\]
 From this we obtain 
\[
k_{i}\circ q=h_{i}\circ\pi_{i}\circ q=h_{i}\circ e_{i}\circ q_{i}=q_{i}
\]
 as required.
\end{proof}
\begin{thm}
\label{thm:Gumm-Henkel}Let $\mathcal{C}$ be a category with finite
sums and kernel pairs. If a functor $F:\mathcal{C}\to\mathcal{C}$
weakly preserves kernel pairs, then it weakly preserves pullbacks
of retractions.
\end{thm}
\begin{proof}
Given the pullback $(P,p_{1},p_{2})$ of retractions $f_{i}:A_{i}\to C,$
we need to show that $(F(P),Fp_{1},Fp_{2})$ is a weak pullback. For
that reason we shall relate it to the kernel pair of $f:=[f_{1},f_{2}]:A_{1}+A_{2}\to C.$

Since the $f_{i}$ are retractions, i.e. right invertible, we can
choose $g_{i}:C\to A_{i}$ with 
\[
f_{i}\circ g_{i}=id_{C}.
\]
 Let $e_{i}:A_{i}\to A_{1}+A_{2}$ be the canonical inclusions, then
\[
f\circ e_{i}=[f_{1},f_{2}]\circ e_{i}=f_{i}.
\]
Define $h_{1}:=[id_{A_{1}},g_{1}\circ f_{2}]$ and $h_{2}:=[g_{2}\circ f_{1},id_{A_{2}}]$,
then $h_{i}:A_{1}+A_{2}\to A_{i}$ satisfy 
\[
h_{i}\circ e_{i}=id_{A_{i}}
\]
 as well as
\[
f_{1}\circ h_{1}=[f_{1}\circ id_{A_{1}},f_{1}\circ g_{1}\circ f_{2}]=[f_{1},f_{2}]=f_{2}\circ h_{2}.
\]
Thus we have established commutativity of the right half of the following
diagram.

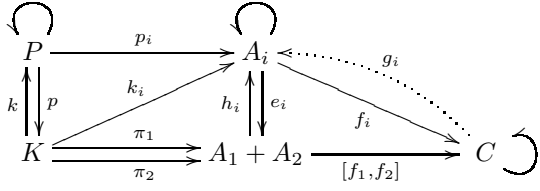
\[
\xymatrix{P\ar@(ur,ul)\ar@<.2pc>[d]^{p}\ar[rr]^{p_{i}} & \, & A_{i}\ar@(ur,ul)\ar@<.2pc>[d]^{e_{i}}\ar@{->}[drr]_{f_{i}}\\
K\ar@<.2pc>[u]^{k}\ar[urr]^{k_{i}}\ar@<.2pc>[rr]^{\pi_{1}}\ar@<-.2pc>[rr]_{\pi_{2}} &  & A_{1}+A_{2}\ar@<.2pc>[u]^{h_{i}}\ar[rr]_{[f_{1},f_{2}]} & \, & {\,\,C\,\,}\ar@(dr,ur)\ar@{..>}@/_{1pc}/[ull]_{g_{i}}
}
\]

We add the kernel pair $(K,\pi_{1},\pi_{2})$ of $f:=[f_{1},f_{2}]$
and the pullback $(P,p_{1},p_{2})$ of $f_{1}$ and $f_{2}$. Now,
the previous lemma asserts that $(K,k_{1},k_{2})$ is a weak pullback
of $f_{1}$ and $f_{2}$, therefore we obtain a morphism $p:P\to K$
with $k_{i}\circ p=p_{i}$.

On the other hand, $(P,p_{1},p_{2})$ being the real pullback, earns
us a unique morphism $k:K\to P$ with $p_{i}\circ k=k_{i}$ and $k\circ p=id_{P}$
by uniqueness.

Next, we apply the functor $F$ to the above diagram. The requirements
of Lemma \ref{lem:sum} remain intact for the image diagram, and assuming
that $F$ weakly preserves the kernel $(K,\pi_{1},\pi_{2})$, we obtain
that $(F(K),Fk_{1},Fk_{2})$ is a weak pullback of $Ff_{1}$ with
$Ff_{2}$. Since furthermore $F(P)$ remains being a retract of $F(K)$
by means of $Fk\circ Fp=Fid_{P}=id_{F(P)},$ we see that $(F(P),Fp_{1},Fp_{2}),$
too, is a weak pullback of $Ff_{1}$ and $Ff_{2}$, as required.
\end{proof}
From this we can obtain our mentioned joint result with Ch. Henkel.
With an elementwise proof this appears in his master thesis, which
has been completed under our guidance \cite{henkel2010}:
\begin{cor}
\label{cor:GH}For a $Set-$endofunctor $F$ the following are equivalent:
\begin{enumerate}
\item $F$ weakly preserves kernel pairs.
\item $F$ weakly preserves pullbacks of epis.
\end{enumerate}
\end{cor}
\begin{proof}
In $Set$, each map $f$ can be factored as $f=\iota\circ e$ where
$e$ is epi and $\iota$ is mono. Therefore, the kernel of $f$ is
the same as the kernel of $e.$ This takes care of the direction $(2\to1).$
For the other direction, the axiom of choice asserts that in $Set$
epis are right invertible, so the conditions of Theorem \ref{thm:Gumm-Henkel}
are met.
\end{proof}

\section{Free-algebra functors and Mal'cev conditions}

Given a finitary algebraic signature $S=\text{(\ensuremath{\mathfrak{f}_{i},n_{i})_{i\in I}}}$,
fixing a family of function symbols $\mathfrak{f}_{i}$, each of arity
$n_{i},$ and given a set $\Sigma$ of equations, let $\mathcal{V}(\Sigma)$
be the variety defined by $\Sigma$, i.e. the class of all algebras
$\mathfrak{A}=(A,(f_{i}^{\mathfrak{A}})_{i\in I})$ satisfying all
equations from $\Sigma.$

The forgetful functor, sending an algebra $\mathfrak{A}$ from $\mathcal{V}(\Sigma)$
to its base set $A,$ has a left adjoint $F_{\Sigma},$ which assigns
to each set $X$, considered as set of variables, the free $\Sigma$-algebra
over $X.$ $F_{\Sigma}(X)$ consists of equivalence classes of terms
$p,$ which arise by syntactically composing basic operations named
in the signature, using only variables from $X$.

Two terms $p$ and $q$ are identified if the equality $p\approx q$
is a consequence of the equations in $\Sigma.$ This is the same as
saying that $p$ and $q$ induce the same operation $p^{\mathfrak{A}}=q^{\mathfrak{A}}$
on each algebra $\mathfrak{A}\in\mathcal{V}(\Sigma).$ Instead of
$p$ we often write $p(x_{1},...,x_{n})$ to mark all occurrences
$x_{1},...,x_{n}$ of variables in the term $p.$

$F_{\Sigma}$ is clearly a functor (in fact a monad), and its action
on maps $\varphi:X\to Y$ can be described as variable substitution,
sending $p(x_{1},...,x_{n})\in F_{\Sigma}(X)$ to $p(\varphi x_{1},...,\varphi x_{n})\in F_{\Sigma}(Y)$.

$\Sigma$ is called \emph{idempotent} if for every function symbol
$\mathfrak{f}$ appearing in $\Sigma$ we have $\Sigma\vdash\mathfrak{f}(x,...,x)\approx x$.
As a consequence, all term operations satisfy the corresponding equations,
so $\Sigma$ is idempotent iff $F_{\Sigma}(\{x\})=\{x\}.$ In this
case, we also call the variety $\mathcal{V}(\Sigma)$ idempotent.
As an example, the variety of all lattices is idempotent, whereas
the variety of groups is not.

In \cite{Gumm2019} we have recently shown that for an idempotent
set $\Sigma$ of equations $F_{\Sigma}$ weakly preserves products.

In 1954, A.I.~Mal'cev \cite{Malcev1954,Malcev1963} discovered that
a variety has \emph{permutable congruences}, i.e. $\Theta\circ\Psi=\Psi\circ\Theta$
holds for all congruences $\Theta$ and $\Psi$ in each algebra of
$\mathcal{V}(\Sigma),$ if and only if there exists a ternary term
$m(x,y,z)$ satisfying the equations 
\[
x\approx m(x,y,y)\text{ and }m(x,x,y)\approx y.
\]
Permutability of congruences $\Theta$ and $\Psi$ can be generalized
to \emph{n-permutability}, requiring that the \emph{n}-fold relational
compositions agree:
\[
\Theta\circ\Psi\circ\Theta\circ\cdots\,=\Psi\circ\Theta\circ\Psi\circ\cdots\,.
\]
 Here each side is meant to be the relational composition of $n$
factors.

J. Hagemann and A.~Mitschke \cite{Hagemann1973}, generalizing the
original Mal'cev result, showed that a variety $\mathcal{V}$ is $n-$permutable
if and only if there are ternary terms $p_{1},...,p_{n-1}$ such the
following series of equations is satisfied:

\begin{eqnarray}
x & \approx & p_{1}(x,y,y)\label{eq:n-permutabl}\\
p_{i}(x,x,y) & \approx & p_{i+1}(x,y,y)\text{ for all }0<i<n-1\nonumber \\
p_{n-1}(x,x,y) & \approx & y.\nonumber 
\end{eqnarray}
Ever since Mal'cev's mentioned result, such conditions, postulating
the existence of derived operations satisfying a (possibly $n-$indexed)
series of equations have been called \emph{Mal'cev conditions}, and
these have played an eminent role in the development of universal
algebra. They may generally be of the form 
\[
\exists\,n\in\mathbb{N}.\,\exists\,p_{1},...,p_{n}.\,\Gamma
\]
 where $p_{1},...,p_{n}$ are terms and $\Gamma$ is a set of equations
involving the terms $p_{1},...,p_{n}.$ Such a condition is supposed
to hold in a variety $\mathcal{V}$ of universal algebras if for some
$n\in\mathbb{N}$ there exist terms $p_{1},...,p_{n}$ in the operations
of $\mathcal{V}$ satisfying the equations in $\Gamma$.

B.~J\'{o}nsson \cite{Jonsson1967} gave a Mal'cev condition characterizing
\emph{congruence distributive} varieties, i.e. varieties $\mathcal{V}$
in which the lattice of congruences of each algebra $\mathfrak{A}\in\mathcal{V}$
is distributive. A Mal'cev condition characterizing \emph{congruence
modular} varieties and involving quaternary terms was found by A.~Day
\cite{day_1969}. In 1981, employing commutator theory, we showed
in \cite{Gumm1981}, how to compose J\'{o}nsson's terms for congruence
distributivity with the original Mal'cev term $m(x,y,z)$ from above
in order to characterize congruence modular varieties, while at the
same time obtaining ternary terms.

Notice, that all Mal'cev conditions mentioned above are idempotent,
i.e. $\Gamma\vdash p_{i}(x,...,x)\approx x$ for each of their terms
$p_{i}.$

\section{Preservation of preimages}

A few years ago, Dent, Kearnes, and Szendrei \cite{Dent2012} invented
a syntactic operation on idempotent sets of equations $\Sigma$, called
the \emph{derivative} $\Sigma'$, and they showed that an idempotent
variety $\mathcal{V}(\Sigma)$ is congruence modular if and only if
$\Sigma\cup\Sigma'$ is inconsistent.

Subsequently, Freese \cite{Freese2013} was able to give a similar
char\-ac\-ter\-i\-za\-tion for \emph{n}-permutable varieties,
using an \emph{order derivative}, which was based on the fact that
a variety is \emph{n}-permutable if and only if its algebras are not
orderable, an insight, said to have been observed already by Hagemann
(unpublished), but which first appears in P.~Selinger's PhD-thesis
\cite{Selinger97}.

It will turn out below, after generalizing the relevant definitions
from \cite{Dent2012} to non-idempotent varieties, that the derivative
also serves us to characterize free-algebra functors preserving preimages.

We start by slightly modifying the definition of weak independence
from \cite{Dent2012}:
\begin{defn}
A term $p$ is \emph{weakly independent} of a variable occurrence
$x,$ if there exists a term $q$ such that $\Sigma\vdash p(x,v_{1},...,v_{n})\approx q(y)$
where $x\ne y$ and $v_{1},...,v_{n}$ are variables. $p$ is \emph{independent}
of $x,$ if $\Sigma\models p(x,z_{1},...,z_{n})\approx p(y,z_{1},...,z_{n})$
where $x,y,$ and all variables $z_{1},...,z_{n}$ are distinct.
\end{defn}
As an example, consider the the variety of groups. The term $m(x,y,z):=xy^{-1}z$
is weakly independent of its first argument, since $m(x,x,y)=xx^{-1}y\approx y$
holds, but it is not independent of the same argument, since $m(x,z_{1},z_{2})=xz_{1}^{-1}z_{2}\not\approx yz_{1}^{-1}z_{2}=m(y,z_{1},z_{2})$.
The term $p(x,y):=xyx^{-1}$ is \emph{independent} of its first argument
in the variety of abelian groups, but not in the variety of all groups,
since $p(x,z)=xzx^{-1}\not\approx yzy^{-1}=p(y,z).$ More generally,
any Mal'cev term $m(x,y,z)$ is \emph{weakly independent} of each
of its arguments, but cannot be independent of either of them. 

Clearly, if $p$ is independent of $x$ then it is also weakly independent
of $x,$ since $p(x,z_{1},...,z_{n})\approx p(y,z_{1},...,z_{n})$
entails $p(x,y,...,y)\approx p(y,y,...,y)=:q(y).$

We now define the \emph{derivative} $\Sigma'$ \emph{of} $\Sigma$
as the set of all equations asserting that a term which is weakly
independent of a variable occurrence $x$ should also be independent
of that variable:

\[
\Sigma':=\{p(x,\overrightarrow{z})\approx p(y,\overrightarrow{z})\mid p(x,\overrightarrow{w})\,\text{weakly independent of }x\}.
\]
We are now ready to state the main theorem of this section:
\begin{thm}
$F_{\Sigma}$ (weakly) preserves preimages if and only if ~$\Sigma\vdash\Sigma'$.
\end{thm}
\begin{proof}
Consider a term $p$ which is weakly independent of a variable occurrence
$x,$ i.e. there exists a term $q$ such that $\Sigma\vdash p(x,w_{1},...,w_{n})\approx q(y)$
where $x,y,w_{1},...,w_{n}$ are occurrences of not necessarily distinct
variables, but $x\ne y.$

We must show that $p(x,z_{1},...,z_{n})\approx p(y,z_{1},...,z_{n})$
for $x,y,z_{1},...,z_{n}$ mutually distinct variables.

Consider the map $\varphi:\{x,y,z_{1},...,z_{n}\}\to\{x,y,w_{1},...,w_{n}\}$
which fixes $x$ and $y$ and sends each $z_{i}$ to $w_{i}$. Then

\begin{eqnarray*}
(F_{\Sigma}\varphi)(p(x,z_{1},...,z_{n})) & = & p(\varphi x,\varphi z_{1},...,\varphi z_{n})\\
 & = & p(x,w_{1},...,w_{n})\\
 & \approx & q(y)\in F_{\Sigma}(\{y\}),
\end{eqnarray*}
so $p(x,z_{1},...,z_{n})$ is in the preimage of $F_{\Sigma}(\{y\})$
under $F_{\Sigma}\varphi.$

The preimage of $\{y\}\subseteq\{x,y,w_{1},...,w_{n}\}$ under $\varphi$
does not contain $x,$ so $\varphi^{-1}(\{y\})\subseteq\{y,z_{1},...,z_{n}\}.$
Assuming that $F_{\Sigma}$ preserves preimages, we obtain 
\[
p(x,z_{1},...,z_{n})\in F_{\Sigma}(\{y,z_{1},...,z_{n}\}),
\]
 so there exists a term $r$ with 
\[
p(x,z_{1},...,z_{n})\approx r(y,z_{1},...,z_{n}).
\]

Since $x,y\notin\{z_{1},...,z_{n}\},$ by substituting $y$ for $x$
in this equation, we also find $p(y,z_{1},...,z_{n})\approx r(y,z_{1},...,z_{n})$,
so 
\[
p(x,z_{1},...,z_{n})\approx p(y,z_{1},...,z_{n})
\]
 by transitivity.

Conversely, assume that $\Sigma\vdash\Sigma'$, then we need to verify
that $F_{\Sigma}$ preserves preimages. According to Lemma \ref{lem:classifying},
we need only verify that $F_{\Sigma}$ preserves classifying preimages.
Let $X$ and $Y$ be disjoint sets and let $\varphi:X\cup Y\to\{x,y\}$
be given, sending elements from $X$ to $x$ and elements from $Y$
to $y.$ Then $Y\subseteq X\cup Y$ is the preimage of $\{y\}$ under
$\varphi$. We must show that the following is a preimage diagram:
\[
\xymatrix{F_{\Sigma}(X\cup Y)\ar[r]^{F_{\Sigma}\varphi} & F_{\Sigma}(\{x,y\})\\
F_{\Sigma}(Y)\ar[r]\ar@{^{(}-->}[u] & F_{\Sigma}(\{y\})\ar@{^{(}->}[u]
}
\]
Given an element $p\in F_{\Sigma}(X\cup Y)$ with $(F_{\Sigma}\varphi)p\in F_{\Sigma}(\{y\})$
we must show that $p\in F_{\Sigma}(Y).$ 

Let $x_{1},...,x_{n},y_{1},...,y_{m}$ be all variables occurring
in $p$ where $x_{i}\in X$ and $y_{i}\in Y$. The assumption yields
\[
p(\varphi x_{1},...,\varphi x_{n},\varphi y_{1},...,\varphi y_{m})=p(x,...,x,y,...,y)\in F_{\Sigma}(y),
\]
i.e. $p(x,...,x,y,...,y)\approx q(y)$ for some term $q\in F_{\Sigma}(\{y\})$.
This means that $p$ is weakly independent of each of its occurrences
of $x,$ so our assumption $\Sigma\vdash\Sigma'$ says that $\mathcal{V}$
also satisfies the equations 
\begin{eqnarray*}
p(x_{1},x_{2},...,x_{n},y_{1},...,y_{m}) & \approx & p(z_{1},x_{2},...,x_{n},y_{1},...,y_{m})\\
 & \approx & p(z_{1},z_{2},...,x_{n},y_{1},...,y_{m})\\
 & \approx & ...\\
 & \approx & p(z_{1},z_{2},...,z_{n},y_{1},...,y_{m})
\end{eqnarray*}
In particular, by choosing $y$ arbitrarily from $\{y_{1},...,y_{m}\},$
we obtain 
\[
p=p(x_{1},...,x_{n},y_{1},...,y_{m})\approx p(y,...,y,y_{1},...,y_{n})\in F_{\Sigma}(Y)
\]
 as required.
\end{proof}
Whereas Dent, Kearnes and Szendrei have shown that congruence modular
varieties are characterized by the fact that adding the derivative
of their defining equations yields an inconsistent theory, we have
just seen that $F_{\Sigma}$ preserving preimages delineates the other
extreme, $\Sigma\vdash\Sigma'$. Clearly, therefore, for modular varieties
$\mathcal{V}(\Sigma)$, the variety functor $F_{\Sigma}$ does not
preserve preimages.

\section{Preservation of kernel pairs}

We begin this section with a positive result:
\begin{thm}
If $\mathcal{V}(\Sigma)$ is a Mal'cev variety, then $F_{\Sigma}$
weakly preserves pullbacks of epis.
\end{thm}
\begin{proof}
According to Corollary \ref{cor:GH}, it suffices to check that $F_{\Sigma}$
weakly preserves kernel pairs of a surjective map $f:X\to Y.$ The
kernel of $f$ is 
\[
K=\{(x,x')\mid f(x)=f(x')\}
\]
 with the cartesian projections $\pi_{1}$ and $\pi_{2}.$

Given $p:=p(x_{1},...,x_{m})\in F_{\Sigma}(X)$ and $q:=q(x_{1}',...,x_{n}')\in F_{\Sigma}(X),$
and $r:=r(y_{1},...,y_{k})\in F_{\Sigma}(Y)$ with 
\[
(F_{\Sigma}f)p=r=(F_{\Sigma}f)q,
\]
we need to find some $\bar{m}\in F_{\Sigma}(K)$ such that $(F_{\Sigma}\pi_{1})\bar{m}=p$
and $(F_{\Sigma}\pi_{2})\bar{m}=q$.

\[
\xymatrix{F_{\Sigma}(K)\ar@<0.15pc>[r]^{F_{\Sigma}\pi_{1}}\ar@<-0.15pc>[r]_{F_{\Sigma}\pi_{2}} & F_{\Sigma}(X)\ar[r]^{F_{\Sigma}f} & F_{\Sigma}(Y)\\
 & \bullet\ar@/^{1pc}/@{..>}[ul]^{\bar{m}}\ar@<0.2pc>[u]^{p}\ar@<-0.2pc>[u]_{q}\ar@/_{1pc}/[ur]_{r}
}
\]

Choose $g$ as right inverse to $f$, i.e. $f\circ g=id_{Y}$, and
define:
\begin{eqnarray*}
\bar{p} & := & p((x_{1},gfx_{1}),...,(x_{m},gfx_{m})),\\
\bar{r} & := & r((gy_{1},gy_{1}),...,(gy_{k},gy_{k})),\\
\bar{q} & := & q((gfx'_{1},x'_{1}),...,(gfx'_{n},x'_{n})),
\end{eqnarray*}
then it is easy to check that $\bar{p},$~$\bar{r,}$ and $\bar{q}$
are elements of $F_{\Sigma}(K),$ hence the same is true for $\bar{m}:=m(\bar{p},\bar{r},\bar{q}),$
where $m(x,y,z)$ is the Mal'cev term. Moreover,

\begin{eqnarray*}
(F_{\Sigma}\pi_{1})\bar{m} & = & (F_{\Sigma}\pi_{1})m(\bar{p},\bar{r},\bar{q})\\
\phantom{{aaaa}} & = & m((F_{\Sigma}\pi_{1})\bar{p},(F_{\Sigma}\pi_{1})\bar{r},(F_{\Sigma}\pi_{1})\bar{q})\\
 & = & m(p(x_{1},...,x_{m}),r(gy_{1},...,gy_{k}),q(gfx_{1}',...,gfx_{n}'))\\
 & = & m(p,(F_{\Sigma}g)r(y_{1},...,y_{k}),(F_{\Sigma}g)(F_{\Sigma}f)q(x_{1}',...,x_{n}'))\\
 & = & m(p,(F_{\Sigma}g)r,(F_{\Sigma}g)r)\\
 & = & p,
\end{eqnarray*}
and similarly, $(F_{\Sigma}\pi_{2})\bar{m}=q.$
\end{proof}
Preservation of kernel pairs for the functor $F_{\Sigma}$ leads to
some interesting syntactic condition on terms, which might be worth
of further study:
\begin{lem}
\label{lem:kernel pres}If $F_{\Sigma}$ weakly preserves kernel pairs,
then for any terms $p,q$ we have:

\[
p(x,x,y)\approx q(x,y,y)
\]
 if and only if there exists a quaternary term $s$ such that 
\begin{eqnarray*}
p(x,y,z) & \approx & s(x,y,z,z)\\
q(x,y,z) & \approx & s(x,x,y,z).
\end{eqnarray*}
\end{lem}
\begin{proof}
The if-direction of the claim is obvious. For the other direction
consider $\varphi,\psi:\{x,y,z\}\to\{x,z\}$ with $\varphi x=\varphi y=\psi x=x,$
and $\varphi z=\psi y=\psi z=z.$ Their pullback is 
\[
P=\{(x,x),(y,x),(z,y),(z,z)\}.
\]
Since $\varphi$ and $\psi$ are epi, according to Theorem \ref{thm:Gumm-Henkel},
their pullback should be weakly preserved, too. Now, given $p$ and
$q$ with $p(x,x,y)\approx q(x,y,y),$ then $(F_{\Sigma}\varphi)p=(F_{\Sigma}\psi)q.$
Therefore, there must be some $s\in F_{\Sigma}(P)$, with $(F_{\Sigma}\pi_{1})s=p$
and $(F_{\Sigma}\pi_{2})s=q$. We obtain:

\begin{eqnarray*}
p(x,y,z) & \approx & (F_{\Sigma}\pi_{1})s((x,x),(y,x),(z,y),(z,z))=s(x,y,z,z),\text{\,\,and}\\
q(x,y,z) & \approx & (F_{\Sigma}\pi_{2})s((x,x),(y,x),(z,y),(z,z))=s(x,x,y,z).
\end{eqnarray*}
\end{proof}
Using the criterion of this Lemma, we now determine, which $n-$permutable
varieties give rise to a functor weakly preserving kernel pairs:
\begin{thm}
If $\mathcal{V}(\Sigma)$ is $n-$permutable, then $F_{\Sigma}$ weakly
preserves kernel pairs if and only if $\mathcal{V}(\Sigma)$ is a
Mal'cev variety, i.e. 2-permutable.
\end{thm}
\begin{proof}
Assuming that $\mathcal{V}(\Sigma)$ is $n-$permutable for $n>2,$
we shall show that it is already $(n-1)-$permutable.

Let $p_{1},...,p_{n-1}$ be the terms from the Mal'cev condition for
$n$-permutability, with the equations \ref{eq:n-permutabl}.

According to Lemma \ref{lem:kernel pres}, we find some term $s(x,y,z,u)$
such that $s(x,y,z,z)\approx p_{1}(x,y,z)$ and $s(x,x,y,z)\approx p_{2}(x,y,z).$
We now define a new term

\[
m(x,y,z):=s(x,y,y,z),
\]
and calculate:

\[
m(x,y,y)\approx s(x,y,y,y)\approx p_{1}(x,y,y)\approx x
\]
as well as

\[
m(x,x,y)\approx s(x,x,x,y)\approx p_{2}(x,x,y)\approx p_{3}(x,y,y).
\]

From this we obtain a shorter chain of equations, discarding $p_{1}$
and $p_{2}$ and replacing them by $m:$

\begin{eqnarray*}
x & \approx & m(x,y,y)\\
m(x,x,y) & \approx & p_{3}(x,y,y)\\
p_{i}(x,x,y) & \approx & p_{i+1}(x,y,y)\text{ for all }2<i<n-1\\
p_{n-1}(x,x,y) & \approx & y.
\end{eqnarray*}
\end{proof}

\section{Conclusion and further work}

We have shown that variety functors $F_{\Sigma}$ preserve preimages
if and only if $\Sigma\vdash\Sigma'$ where $\Sigma'$ is the derivative
of the set of equations $\Sigma$.

For each Mal'cev variety, $F_{\Sigma}$ weakly preserves kernel pairs.
For the other direction, if $F_{\Sigma}$ weakly preserves kernel
pairs, then every equation of the shape $p(x,x,y)\approx q(x,y,y)$
ensures the existence of a term $s(x,y,z,u),$ such that $p(x,y,z)\approx s(x,y,z,z)$
and $q(x,y,y)\approx s(x,x,y,z).$

This intriguing algebraic condition appears to be new and certainly
deserves further study from a universal algebraic perspective. Adding
it to Freese's char\-ac\-ter\-i\-za\-tion of $n$-permutable
varieties by means of his ``order-derivative'', allows to distinguish
between the cases $n=2$ (Mal'cev varieties) and $n>2$, which cannot
be achieved using his order derivative, alone.

It would be desirable to see if this consequence of weak kernel preservation
can be turned into a concise if-and-only-if statement.

\bibliographystyle{plain}

\end{document}